\documentclass[11pt,a4paper]{amsart}
\usepackage{amssymb,amsmath,epsfig,graphics,mathrsfs}

\usepackage{graphicx}

\usepackage{fancyhdr}
\usepackage{hyperref}
\hypersetup{
 colorlinks   = true,
 urlcolor     = blue,
 linkcolor    = blue,
 citecolor   = red ,
 bookmarksopen=true
}

\usepackage{amsmath}
\usepackage{amsfonts}
\usepackage{amssymb}
\usepackage{amsthm}
\usepackage{epsfig,graphics,mathrsfs}
\usepackage{graphicx}

\usepackage{latexsym} 
\usepackage{amssymb,amsmath}
\usepackage{amsthm}
\usepackage{longtable,booktabs,setspace} 
\usepackage{url}

\usepackage[usenames, dvipsnames]{color} 

\usepackage{hyperref}
\usepackage{verbatim}

\newcommand*\MSC[1][1991]{\par\leavevmode\hbox{%
\textit{\,\,\,\,\, #1 Mathematical subject classification:\ }}}
\newcommand\blfootnote[1]{%
  \begingroup
  \renewcommand\thefootnote{}\footnote{#1}%
  \addtocounter{footnote}{-1}%
  \endgroup
}

\def \phi {\varphi}

\def \R {\mathbb{R}}

\def \G{\Gamma}

\newcommand{\Rn}{\mathbb R^n}

\newcommand{\Hn}{\mathbb H^n}

\newcommand{\p}{\partial}

\newcommand{\bG}{\mathbb {G}}
\newcommand{\bg}{\mathfrak g}

\newcommand{\la}{\lambda}

\numberwithin{equation}{section}

\newcommand{\beq}{\begin{equation}}
\newcommand{\bea}[1]{\begin{array}{#1} }
\newcommand{\eeq}{ \end{equation}}
\newcommand{\ea}{ \end{array}}

\newcommand{\ve}{\varepsilon}

\newcommand{\Lp}{L^p}

\newtheorem{theorem}{Theorem}[section]
\newtheorem{lemma}[theorem]{Lemma}

\newtheorem{corollary}[theorem]{Corollary}

\numberwithin{equation}{section}

\begin{document}

\title{A heat equation approach to intertwining}

{\blfootnote{\MSC[2020]{35K08, 35R11, 53C18}}}


\date{}

\begin{abstract}
In this paper we present a new approach based on the heat equation and extension problems to some intertwining formulas arising in conformal CR geometry.   
\end{abstract}

\author{Nicola Garofalo}

\address{Dipartimento d'Ingegneria Civile e Ambientale (DICEA)\\ Universit\`a di Padova\\ Via Marzolo, 9 - 35131 Padova,  Italy}
\vskip 0.2in
\email{nicola.garofalo@unipd.it}

\thanks{Both authors are supported in part by a Progetto SID: ``Non-local Sobolev and isoperimetric inequalities", University of Padova, 2019.}

\author{Giulio Tralli}
\address{Dipartimento d'Ingegneria Civile e Ambientale (DICEA)\\ Universit\`a di Padova\\ Via Marzolo, 9 - 35131 Padova,  Italy}
\vskip 0.2in
\email{giulio.tralli@unipd.it}

\maketitle

\section{Introduction}\label{S:intro} 

In this paper we present a heat semigroup approach to some intertwining formulas which arise in conformal CR geometry. The prototypical example of the questions we are interested in can be described as follows. In $\Rn$ with $n\geq 2$, for $0<s<1$ consider the pseudodifferential operator which in Fourier transform is given by $\widehat{(-\Delta)^s u}(\xi) = (2\pi |\xi|)^{2s} \hat u(\xi)$,
see \cite{R, St}. Then, for every $x\in\R^n$, and each fixed $y>0$, one has the following intertwining relation
\begin{equation}\label{Eu}
\left(-\Delta\right)^s\left((|x|^2 + y^2)^{-\frac{n- 2s}2}\right)= \frac{\G\left(\frac{n}{2}+s\right)}{\G\left(\frac{n}{2}-s\right)}  \left(2 y\right)^{2s} (|x|^2 + y^2)^{-\frac{n+ 2s}2}. 
\end{equation}
The objective of this note is to use the heat equation to establish a suitable variant  of \eqref{Eu} in which $(-\Delta)^s$ is replaced  by the nonlocal conformal horizontal Laplacian $\mathscr L_s$ on a Lie group of Heisenberg type $\bG$. It is worth mentioning here that our approach is  new even in the classical case of \eqref{Eu}.

To provide the reader with some historical perspective we recall that  \eqref{Eu} implicitly appeared in the celebrated 1983 work of Lieb concerning the best constants in the Hardy-Littlewood-Sobolev inequalities, see (3.12), (3.13) in the proof of \cite[Theorem 3.1]{L}. In this connection we observe that a notable consequence of \eqref{Eu} is that it provides a family of positive solutions to the nonlocal Yamabe equation in $\R^n$ 
\begin{equation}\label{yams}
\left(-\Delta\right)^su= u^{\frac{n+2s}{n-2s}}.
\end{equation} 
The uniqueness of such positive solutions was subsequently proved by Chen, Li, and Ou in \cite{CLO}. A different way of looking at \eqref{Eu}, which also underscores its geometric invariance, goes back to the ideas in \cite{B, Bsur}. For a function $f$ on $\Rn$ denote by $f^\star$ its push forward to $\mathbb S^n\subset \R^{n+1}$ through the stereographic projection. Branson characterized the conformal pseudodifferential operators $A_s$ of order $2s$ on $\mathbb S^n$ satisfying the intertwining formula
\begin{equation}\label{branson3}
A_{s}\left(\left(\left(\frac{1+|\cdot|^2}2\right)^{\frac n2 - s} f\right)^\star\right) = \left(\left(\frac{1+|\cdot|^2}2\right)^{\frac n2 + s} (-\Delta)^s f\right)^\star,
\end{equation}
and proved that they have the following spectral representation
\begin{equation}\label{branson}
A_{s} = \frac{\G(\sqrt{-\Delta_{\mathbb S^n} + \frac{(n-1)^2}{4}} + s + \frac 12)}{\G(\sqrt{-\Delta_{\mathbb S^n} + \frac{(n-1)^2}{4}} - s + \frac 12)}.
\end{equation}
The number $A_{s}(1) = \frac{\G(\frac n2 + s)}{\G(\frac n2 - s)}$ is called, up to a renormalising constant, the fractional $Q_s$-curvature of $\mathbb S^n$. For more insights on these aspects we refer the reader to \cite{BJ, CG, dmG} .
We note that from \eqref{branson3} it is immediate to recognise the validity of \eqref{Eu}. It suffices to take $f(x) = \big(\frac{1+|x|^2}2\big)^{s-\frac n2}$, and observe that
\[
(-\Delta)^s \big(\frac{2}{1+|x|^2}\big)^{\frac n2-s} = \big(\frac2{1+|x|^2}\big)^{\frac n2 + s} \big(A_{s}(1^\star)\big)^\star = \big(\frac2{1+|x|^2}\big)^{\frac n2 + s} A_{s}(1) = \frac{\G(\frac n2 + s)}{\G(\frac n2 - s)} \big(\frac2{1+|x|^2}\big)^{\frac n2 + s}.
\] 

In our main result, Theorem \ref{CHrevisited} below, we present a new approach to intertwining formulas such as \eqref{Eu} which is exclusively based on the heat equation. 
To explain the main ideas without delving into technical aspects we mention that, instead of looking at $(-\Delta)^s$, we consider the nonlocal heat operator $(\p_t - \Delta)^s$ and its \emph{extension problem}: given a function $u\in C^\infty_0(\Rn_x\times \R_t)$, find $U\in C^\infty(\Rn_x\times \R_t\times \R^+_y)$ such that
\begin{equation}\label{exths0}
\begin{cases}
\mathfrak P^{(s)} U \overset{def}{=} \frac{\p^2 U}{\p y^2}  + \frac{1-2s}y \frac{\p U}{\p y} + \Delta_x U - \frac{\p U}{\p t} = 0,
\\ 
U(x,t,0) = u(x,t).
\end{cases}
\end{equation}
We claim that the fundamental solution $q^{(s)}(x,y,t)$ of the operator $\mathfrak P^{(s)}$ in \eqref{exths0} can be used to provide a simple proof of \eqref{Eu}.  
To see this we note that, if $w\in \R^{2(1-s)}$ and $y = |w|$, then $\mathfrak P^{(s)}$ represents the action on functions $U(x,w,t) = \overline U(x,y,t)$ of the heat operator $\Delta_x +\Delta_w  - \p_t$ in the space with fractal dimension $\R^{n+2(1- s)}\times  \R^+_t$ whose fundamental solution (with pole at the origin) is given by
\begin{equation}\label{Euqs}
 q^{(s)}(x,y,t) = \frac{1}{(4\pi t)^{\frac{n}{2}+1- s}}e^{-\frac{|x|^2+y^2}{4t}}.
\end{equation}
If we denote by $q^{(-s)}(x,y,t)$ the heat kernel obtained by replacing $s$ into $-s$ in \eqref{Euqs}, then by Bochner's subordination the two functions  
\begin{equation}\label{Es}
E^{(\pm s)}(x,y)\overset{def}{=}\int_0^{\infty} q^{(\pm s)}(x,y,t) dt
\end{equation}
are the fundamental solutions of the time-independent differential operators $\frac{\p^2}{\p y^2}  + \frac{1\mp 2s}y \frac{\p}{\p y} + \Delta_x$. If we now use the elementary (but very important) consequence of \eqref{Euqs} that
\begin{equation}\label{stic}
\int_0^{\infty} q^{(\pm s)}(x,y,t) dt = \frac{\G(\frac{n\mp 2s}2)}{4 \pi^{\frac n2+1 \mp s}} (|x|^2 + y^2)^{-\frac{n\mp 2s}2},
\end{equation}
then in view of \eqref{Es} the problem of proving \eqref{Eu} is  reduced to that of establishing the equivalent dimension-free relation
\begin{equation}\label{Eunodim}
\left(-\Delta\right)^s\left(E^{(s)}(\cdot,y)\right)(x)=\left(2\pi y\right)^{2s} E^{(-s)}(x,y).
\end{equation}
To prove \eqref{Eunodim} we once again bring to the center stage the heat semigroup $P_t = e^{-t\Delta}$. If, in fact, instead of the Fourier transform definition of $(-\Delta)^s$ we use the (equivalent) one given by the following formula of Balakrishnan \cite{Bala},
\begin{equation}\label{bala}
(-\Delta)^s u(x) = -\frac{s}{\G(1-s)} \int_0^\infty \frac{1}{t^{1+s}} (P_t u(x) - u(x)) dt,
\end{equation}
then the proof of \eqref{Eunodim} hinges on the computation of $P_t (E^{(s)}(\cdot,y))(x)$, with $E^{(s)}$ defined by \eqref{Es}. As we show in \eqref{chko} below and subsequent considerations, the Chapman-Kolmogorov identity and an elementary change of variable  allow to easily complete this computation and establish \eqref{Eunodim}. 
By comparison, a direct proof of \eqref{Eu} via Fourier transform rests on elaborate computations involving special functions which tend to overshadow the geometric content of the formula itself (for complete details see \cite[Lemma 8.6]{Gft}).

The pseudodifferential operators $(-\Delta)^s$ and  \eqref{branson} have a counterpart in CR geometry. In this context the Heisenberg group $\Hn \cong \mathbb C^{n}\times\mathbb R$ with coordinates $(z,\sigma)$\footnote{we explicitly mention here that traditionally the letter $t$ is reserved for the vertical variable in $\Hn$. However, since we want to indicate the time variable with $t$, we have opted for the notation $\sigma$.} occupies a special position since, via the Cayley transform, such group has a conformal identification with the boundary of the unit ball in $\mathbb C^{n+1}$, the sphere $\mathbb S^{2n+1}$ with its standard CR structure. Analogously to the term $(|x|^2+1)^{-1}$ in the stereographic projection, the function $|i+ (4\sigma + i|z|^2)|^{-2} = ((|z|^2+1)^2+16\sigma^2)^{-1}$ appears as the conformal factor in the Cayley transform, see \cite[Section 4]{JL}, and in fact appropriate powers of such factor played a central role in the works of Jerison and Lee \cite{JL2} and Frank and Lieb \cite{FL}. In $\Hn$, with $T = \p_\sigma$, the CR conformal nonlocal operator $\mathscr L_s$ was first introduced in the work by Branson et al. \cite[see (1.33)]{BFM} via the spectral formula 
\begin{equation}\label{bfm}
\mathscr L_s = 2^s |T|^s \frac{\G(-\frac 12\mathscr L |T|^{-1} + \frac{1+s}2)}{\G(-\frac 12 \mathscr L |T|^{-1} + \frac{1-s}2)},\ \ \ \ \ \ \ 0<s<1,
\end{equation}
where we have denoted by $\mathscr L$ the Kohn-Spencer horizontal Laplacian in $\Hn$.
We note that, when $0<s<1$ the pseudodifferential operator $\mathscr L_s$ is dramatically different from $\mathscr L^s = (-\mathscr L)^s$ (for the definition of this operator see \eqref{balaL} below), and these two operators coincide only in the local case $s=1$. In fact, by formally letting $s\nearrow 1$ in \eqref{bfm}, and using $\G(x+1) = x\G(x)$, we obtain $\mathscr L_1 = -\mathscr L = \mathscr L^1$.
In their work \cite{FGMT} Frank et al. introduced a new extension problem for the nonlocal operator \eqref{bfm}, very different from that for its non-geometric counterpart $\mathscr L^s$ in \cite{FF} (based on the work of Caffarelli and Silvestre \cite{CS}), and used scattering theory to solve it. Subsequently, Roncal and Thangavelu employed a parabolic version of the extension problem in \cite{FGMT} to establish some optimal Hardy inequalities in $\Hn$ \cite{RTaim},  or more in general in groups of Heisenberg type  \cite{RT}. Such time-dependent extension problem plays an important role in our recent work \cite{GTfeel}, as well as in the present paper which can be seen as a continuation of such work.

To introduce the relevant geometric framework we recall that Lie groups of Heisenberg type were introduced by Kaplan \cite{Ka} in connection with hypoellipticity questions. They are geometrically interesting since on one hand they retain most of the important symmetries of $\Hn$, on the other they naturally arise as the nilpotent component $N$ in the Iwasawa decomposition $\bG = KAN$ of a simple group of rank one, see \cite{CDKR}.
In a group of Heisenberg type $\bG$ (see Section \ref{S:confi} for the relevant definitions) the following generalisation of \eqref{bfm} was introduced in \cite{RT}. Let $\mathfrak g = V_1\oplus V_2$ be the Lie algebra of $\bG$, and denote $m = \operatorname{dim} V_1$, $k = \operatorname{dim} V_2$ (we note that the complex structure of $\bG$ forces $m=2n$ for some $n\in \mathbb N$). We routinely identify $\mathfrak g$ with $\R^m\times\R^k$, and the generic point $g\in \bG$ with its logarithmic coordinates $(z,\sigma)\in \R^m\times\R^k$. Let $\mathscr L$ be a given horizontal Laplacian in $\bG$ associated with an orthonormal basis of the horizontal layer $V_1$. Consider the pseudo-differential operator of order $2s$ defined by 
\begin{equation}\label{bfmH}
\mathscr L_s = 2^s (-\Delta_\sigma)^{s/2} \frac{\G(-\frac 12\mathscr L (-\Delta_\sigma)^{-1/2} + \frac{1+s}2)}{\G(-\frac 12 \mathscr L (-\Delta_\sigma)^{-1/2} + \frac{1-s}2)},\ \ \ \ \ \ 0<s<1.
\end{equation}
We note that for $\Hn$ the dimension of the vertical layer is $k=1$ and  $(-\Delta_\sigma)^{s/2}$ in \eqref{bfmH} becomes $|T|^s$, thus giving back \eqref{bfm}. 
In a group of Heisenberg type $\bG$ the counterpart of the intertwining formula \eqref{Eu} is given by the following 
\begin{align}\label{Gnoneu}
& \mathscr L_s\left(((|z|^2+y^2)^2+16|\sigma|^2)^{-\frac{m + 2k - 2s}4}\right) 
\\
&=\frac{\G\left(\frac{m+2+2s}{4}\right)\G\left(\frac{m+2k+2s}{4}\right)}{\G\left(\frac{m+2-2s}{4}\right)\G\left(\frac{m+2k-2s}{4}\right)} (4y)^{2s}((|z|^2+y^2)^2+16|\sigma|^2)^{-\frac{m + 2k + 2s}4},
\notag
\end{align}
for $(z,\sigma)\in\bG$, and $y>0$. We stress that in the particular case of the Heisenberg group $\mathbb{H}^n$ (which corresponds to the case $m = 2n$ and $k=1$) the function appearing in the left-hand side of \eqref{Gnoneu} defines, up to group translations, the unique extremal of the Hardy-Littlewood-Sobolev inequalities obtained by Frank and Lieb in \cite{FL}. We also note that, similarly to \eqref{yams} above, a remarkable by-product of \eqref{Gnoneu} is that up to a renormalising factor the function $u(z,\sigma) = ((|z|^2+y^2)^2+16|\sigma|^2)^{-\frac{Q - 2s}4}$ (where $Q=m+2k$ denotes the so-called homogeneous dimension of the group $\bG$) provides a positive solution to the following nonlocal Yamabe equation
\begin{equation}\label{cryf}
\mathscr L_s u = u^{\frac{Q+2s}{Q-2s}}.
\end{equation}
Following \cite{FGMT}, one can consider semilinear equations as \eqref{cryf} in a suitable class of CR manifolds having $\mathbb{H}^n$ as (flat-)model case: proving existence of positive solutions in this framework would in fact resolve the fractional CR Yamabe problem posed in \cite[p. 103-104]{FGMT}. We also mention \cite{GMM, Kr} for multiplicity results for solutions of \eqref{cryf} with unrestricted sign in $\mathbb{H}^n$.

In connection with \eqref{Gnoneu} we recall that the existence of intertwining operators in semisimple Lie groups is known since the pioneering work \cite{KS} by Knapp and Stein. A systematic treatment of intertwining kernels in  Lie groups of Heisenberg type was developed by Cowling in \cite{Co} and subsequently by Cowling and Haagerup in \cite{CH}. 
Formula \eqref{Gnoneu} was recently proved by Roncal and Thangavelu in \cite[Theorem 3.1]{RTaim} and \cite[Theorem 3.7]{RT}, where they used it to find a M. Riesz type inverse of the operator $\mathscr L_s$. Their approach relies on non-commutative Fourier analysis and group representation theory, and it is inspired to the results in \cite[Section 3]{CH}.

Our approach to the CR intertwining formula \eqref{Gnoneu}  is inspired to the above described strategy leading to \eqref{Eunodim}. While we refer the reader to Section \ref{S:confi} for the relevant details and a description of the background results from \cite{FGMT}, \cite{RTaim}, \cite{GTfeel} and \cite{GTstep2}, here we mention that the main step is to consider the parabolic extension problem referred to above:
given a function $u\in C^\infty_0(\bG\times \R_t)$, find a function $U\in C^\infty(\bG\times  \R_t \times \R^+_y)$ such that
\begin{equation}\label{parext0}
\begin{cases}
\mathfrak P_{(s)} U \overset{def}{=} \frac{\p^2 U}{\p y^2}  + \frac{1-2s}y \frac{\p U}{\p y} + \frac{y^2}4 \Delta_\sigma U+ \mathscr L U - \frac{\p U}{\p t} = 0,\ \ \ \ \ \text{in}\ \bG\times  \R_t \times \R^+_y,
\\
U(g,t,0) = u(g,t).
\end{cases}
\end{equation}
In our recent work \cite{GTfeel} we  have proved that the fundamental solution (with pole at the origin) of the operator $\mathfrak P_{(s)}$ in \eqref{parext0} is the function in the thick space $\bG\times \R^{2(1- s)} \times \R^+$
given by 
\begin{align}\label{parBfs0}
 q_{(s)}((z,\sigma),t,y) & =  \frac{2^k}{(4\pi t)^{\frac{m}2 +k +1-s}} \int_{\R^k} e^{- \frac it \langle \sigma,\la\rangle}   \left(\frac{|\la|}{\sinh |\la|}\right)^{\frac m2+1-s} e^{-\frac{|z|^2 +y^2}{4t}\frac{|\la|}{\tanh |\la|}} d\la.
\end{align}
Similarly to the case of $\Rn$, the conformal CR invariants of the formula \eqref{Gnoneu} are embedded in the fundamental solution $q_{(s)}((z,\sigma),t,y)$ in \eqref{parBfs0}. This claim will follow from two basic results stated below: Theorem \ref{CHrevisited} and Theorem \hyperref[A]{A}. The former is the main result in this paper, the latter represents a counterpart of \eqref{stic} and was proved in \cite[Theorem 1.4]{GTfeel}. From \eqref{parBfs0} and Bochner's principle of subordination we know that the distribution
\begin{equation}\label{frake0}
\mathfrak e_{(s)}((z,\sigma,y) \overset{def}{=} \int_0^\infty q_{(s)}((z,\sigma),t,y) dt
\end{equation}
is the fundamental solution with pole at the origin of the time-independent part of $\mathfrak P_{(s)}$, i.e., the conformal extension operator
$\mathfrak L_{(s)} = \frac{\p^2}{\p y^2} + \frac{1-2s}{y} \frac{\p }{\p y} + \frac{y^2}4 \Delta_\sigma + \mathscr L.
$
The reader should bear in mind that $\mathfrak e_{(s)}((z,\sigma,y)$ is the CR counterpart of the function $E^{(s)}(x,y)$ in  \eqref{Es} above.  Along with $\mathfrak e_{(s)}((z,\sigma),y)$, we consider the distribution $\mathfrak e_{(-s)}((z,\sigma),y)$ obtained by changing $s$ into $-s$ in  \eqref{parBfs0} and \eqref{frake0}. We are finally ready to state our main result.
 
\begin{theorem}[Geometric intertwining]\label{CHrevisited}
Let $\bG$ be a group of Heisenberg type and let $s\in (0,1)$. For every $g\in\bG$ and $y>0$ one has
\begin{equation}\label{confCH}
\mathscr L_s (\mathfrak e_{(s)}(\cdot,y))(g) = (2\pi y)^{2s} \mathfrak e_{(- s)}(g,y).
\end{equation}
\end{theorem}
Theorem \ref{CHrevisited} is the CR counterpart of the dimension-free identity \eqref{Eunodim}. It is worth mentioning here that in our proof of \eqref{confCH} we do not use the definition
\eqref{bfmH} of $\mathscr L_s$ since the latter would immediately lead into the elaborate machinery of non-commutative Fourier analysis and group representation theory. Instead, we base our analysis on the equivalent representation \eqref{RTs} below which is only formally similar to \eqref{bala}.

With Theorem \ref{CHrevisited} in hands, we combine it with the following result, which is \cite[Theorem 1.4]{GTfeel}, that further underscores the geometric relevance of the functions $\mathfrak e_{(\pm s)}((z,\sigma),y)$.

\medskip

\noindent \textbf{Theorem A.}\label{A}\  
\emph{Let $0<s\le 1$. In any group of Heisenberg type $\bG$ the distribution in the thick space $\bG\times \R^+_y$ defined by \eqref{frake0} is given by
\begin{equation}\label{qmasomeno0}
\mathfrak e_{(s)}((z,\sigma),y) = \frac{\G(s)}{(4\pi)^{1-s}} C_{(s)}(m,k)\ ((|z|^2+y^2)^2+16|\sigma|^2)^{-\frac 12(\frac m2 + k - s)},
\end{equation}
where we have let
\begin{equation}\label{C}
C_{(s)}(m,k) = \frac{2^{\frac m2 + 2k-3s-1} \G(\frac 12(\frac m2+1-s)) \G(\frac 12(\frac m2 + k -s))}{\pi^{\frac{m+k+1}2} \G(s)}.
\end{equation}
An equation similar to \eqref{qmasomeno0} holds if we replace $s$ with $-s$, provided that $\G(s)$ is replaced by $|\G(-s)|$}.
\medskip

It should now be clear to the reader that the combination of Theorem \ref{CHrevisited} and Theorem \hyperref[A]{A} immediately implies the intertwining formula \eqref{Gnoneu}, see Corollary \ref{C:chsect} below.

In closing, we provide a brief description of the paper. In Section \ref{S:nci} we present a fairly general non-geometric version of the intertwining  phenomenon that goes well beyond the setting of Lie groups of Heisenberg type. We warn the reader that, although on a formal level the statement  of Theorem  \ref{CHnon} seems similar that of Theorem \ref{CHrevisited}, the geometric content of the latter result is lost in the former.  Our intent is to illustrate the flexibility of the heat equation approach which allows to treat situations in which the arsenal of Fourier analysis is not readily available. For this we have chosen the framework of stratified nilpotent Lie groups (aka Carnot groups) of arbitrary step, but the proof of Theorem  \ref{CHnon} holds with no changes for the fractional powers of the infinitesimal generator of a Dirichlet form under by now standard assumptions which guarantee Gaussian estimates for the relevant heat semigroup. Section \ref{S:confi} is devoted to proving Theorem  \ref{CHrevisited}. We begin by explaining in some detail the geometric extension problem \eqref{parext0}, and we introduce the two operators $P_{(\pm s),t}$. We highlight the new difficulties that one encounters with respect to the non-geometric case and we establish two preliminary results of an elementary character, Lemma \ref{kernallazzamentoy} and  Lemma \ref{L:cowboy}, that serve to overcome such obstructions. Finally, we prove Theorem  \ref{CHrevisited}. In Corollary \ref{C:chsect}, which closes the paper, we combine this result with Theorem \hyperref[A]{A}  to establish \eqref{Gnoneu}.


\section{Non-conformal intertwining formula}\label{S:nci}

In this section we establish in the setting of stratified nilpotent Lie groups of arbitrary steps (aka Carnot groups) a general version of the intertwining formula \eqref{Eunodim}, see Theorem \ref{CHnon} below. On one hand, our purpose is to show that the heat equation approach outlined in the case of flat $\Rn$ in the introduction works in great generality. On the other hand, with the present section we intend to highlight the dramatic differences between the geometric and the non-geometric operators $\mathscr L_s$ and $\mathscr L^s$. To introduce our discussion we recall that a Carnot group of step $r$ is a simply-connected Lie group $\bG$ whose Lie algebra $\bg$ is stratified and $r$-nilpotent. This means that $\bg = V_1 \oplus ... \oplus V_r$, with $[V_1,V_i] = V_{i+1}$, $i=1,...,r-1$, and $[V_1,V_r] = \{0\}$. The bracket generating layer $V_1$ of $\bg$ is called the horizontal layer. If we denote by $L_g(g') = g\circ g'$ the operator of left-translation on $\bG$, and define left-invariant vector fields $X_j(g) = d L_g(e_j)$, where $\{e_1,...,e_m\}$ is an orthonormal basis of $V_1$, then a horizontal Laplacian in $\bG$ (associated with the basis $\{e_1,...,e_m\}$) is defined by $\mathscr L = \sum_{j=1}^m X_j^2$. Such differential operator fails to be elliptic at every point $g\in \bG$, but according to H\"ormander's theorem in \cite{Ho} it is hypoelliptic in $\bG$. In \cite{Fo75} Folland proved the existence of a strictly positive kernel $p(g,g',t)$, $C^\infty$ off the diagonal, such that the heat semigroup $P_t = e^{-t\mathscr L}$ is represented by $P_t f(g) = \int_{\bG} p(g,g',t) f(g') dg'$. Such semigroup is stochastically complete ($P_1 1 = 1$) and a contraction on $\Lp(\bG)$ for every $1\le p\le \infty$. Furthermore, the kernel $p(g,g',t)$ satisfies the Gaussian estimates
\begin{equation}\label{ge}
C t^{-Q/2} e^{-\alpha \frac{d(g,g')^2}{t}}\le p(g,g',t) \le C^{-1} t^{-Q/2} e^{-\beta \frac{d(g,g')^2}{t}},
\end{equation}
for appropriate universal constants $C, \alpha, \beta>0$. In \eqref{ge} we have indicated with $Q = \sum_{j=1}^r j \operatorname{dim} V_j$ the homogeneous dimension of $\bG$ with respect to the anisotropic dilations associated with the grading of the Lie algebra, whereas $d(g,g')$ denotes the control distance, see \cite{Fo75, VSC}.
The semigroup $P_t$ is all that is needed to introduce the fractional powers $\mathscr L^s = (-\mathscr L)^s$ by means of the following formula due to Balakrishnan 
\begin{equation}\label{balaL}
\mathscr L^s u(g) = -\frac{s}{\G(1-s)} \int_0^\infty \frac{1}{t^{1+s}} (P_t u(g) - u(g)) dt.
\end{equation}
 Next, we consider the parabolic extension problem for the pseudodifferential operator $(\p_t - \mathscr L)^s$, see \cite{Gejde}: given a function $u\in C^\infty_0(\bG\times \R_t)$, find $U\in C^\infty(\bG\times \R_t\times \R^+_y)$ such that
$$
\begin{cases}
\mathfrak P^{(s)} U \overset{def}{=} \frac{\p^2 U}{\p y^2}  + \frac{1-2s}y \frac{\p U}{\p y} + \mathscr L U - \frac{\p U}{\p t} = 0,
\\ 
U(g,t,0) = u(g,t).
\end{cases}
$$
Since we can think of the differential operator $\mathfrak P^{(s)}$ as a heat operator on the group $\bG\times \R^{2(1-s)}_w\times \R^+_t$ acting on functions $U(g,w,t) = \overline U(g,y,t)$, with $y = |w|$, its fundamental solution with pole at the identity is given by
\begin{equation}\label{defq}
 q^{(s)}(g,t,y)  =   \frac{1}{(4\pi t)^{1- s}}e^{-\frac{y^2}{4t}} p(g,e,t),
\end{equation}
where we have denote by $e\in \bG$ the identity element. We stress that the function $q^{(s)}(g,t,y) $ in \eqref{defq} is dramatically different from its geometric counterpart 
 $q_{(s)}((z,\sigma),t,y)$ in \eqref{parBfs0} above which is not a product. 
Hereafter, we indicate with the $q^{(-s)}(g,t,y)$ the function obtained by changing $s$ into $-s$ in \eqref{defq} (see \eqref{qmenos} below), and we define
\begin{equation}\label{defe}
\mathfrak e^{(\pm s)}(g,y) \overset{def}{=} \int_0^\infty q^{(\pm s)}(g,t,y) dt.
\end{equation}
It should be clear to the reader that the function $\mathfrak e^{(s)}$ in \eqref{defe} is the fundamental solution (with pole at the identity) of the time-independent extension operator $\frac{\p^2}{\p y^2}  + \frac{1-2s}y \frac{\p}{\p y} + \mathscr L$.
We have the following result.

\begin{theorem}[Non-geometric intertwining]\label{CHnon}
Let $\bG$ be a Carnot group and $s\in (0,1)$. For any $g\in\bG$ and $y>0$ one has
\begin{equation}\label{intrallnongeom}
\mathscr L^s (\mathfrak e^{(s)}(\cdot,y))(g) = (2\pi y)^{2s} \mathfrak e^{(- s)}(g,y).
\end{equation}
\end{theorem}

\begin{proof}
First, we observe that if $u$ is a $C^\infty(\bG)$ function which suitably decays at infinity, then
the estimates \eqref{ge} allow to integrate by parts in \eqref{balaL} obtaining 
\begin{equation}\label{bala2}
\mathscr L^s u(g) = - \frac{s}{\G(1-s)} \int_0^\infty \frac{1}{\tau^{1+s}} \int_0^\tau \p_t P_t u(g) dt d\tau = - \frac{1}{\G(1-s)} \int_0^\infty t^{-s} \p_t P_t u(g) dt.
\end{equation}
If we apply $P_t$ to $u(g) = \mathfrak e^{(s)}(g,y)$, and we use \eqref{defe}, we find
\begin{align}\label{chko}
P_t \mathfrak e^{(s)}(g,y)& =  \int_{\bG} p(g,g',t) \int_0^\infty \frac{1}{(4\pi \tau)^{1- s}}e^{-\frac{y^2}{4\tau}} p(g',0,\tau) d\tau dg'
\\
& = \int_0^\infty \frac{1}{(4\pi \tau)^{1- s}}e^{-\frac{y^2}{4\tau}} \int_{\bG} p(g,g',t)p(g',e,\tau) dg' d\tau 
\notag
\\
& = \int_0^\infty \frac{1}{(4\pi \tau)^{1- s}}e^{-\frac{y^2}{4\tau}} p(g,e,t+\tau) d\tau,  
\notag
\end{align}
where in the last equality we have used the Chapman-Kolmogorov equation (semigroup property for $P_t$)
\[
p(g,g'',t+\tau) = \int_{\bG} p(g,g',t) p(g',g'',\tau) dg' \quad\mbox{ for }g''\in\bG.
\]
Substituting this formula in \eqref{bala2}, we find 
\begin{align*}
\mathscr L^s(\mathfrak e^{(s)}(\cdot,y))(g)  & = - \frac{(4\pi)^{(s-1)}}{\G(1-s)} \int_0^\infty   \int_0^\infty \left(\frac{\tau}t\right)^{s} \p_t\left( e^{-\frac{y^2}{4\tau}} p(g,e,t+\tau) \right) \frac{d\tau}\tau dt.
\end{align*}
The change of variable
\begin{equation}\label{change}
(v,\rho)= \left(t + \tau,\frac{\tau}t\right),\qquad\qquad \text{or}\qquad\qquad (t,\tau) = \left(\frac{v}{1+\rho},\frac{v\rho}{1+\rho}\right),
\end{equation}
for which $dt d\tau = \frac{v}{(1+\rho)^2} dv d\rho$ and $ \p_t = \p_v - \frac{\rho(1+\rho)}{v}\p_\rho$, now gives
$$\mathscr L^s (\mathfrak e^{(s)}(\cdot,y))(g) = \frac{-(4\pi)^{s-1}}{\G(1-s)} \int_0^\infty\int_0^\infty \frac{\rho^{s-1}}{1+\rho}\left(\frac{\partial}{\partial v} - \frac{\rho(1+\rho)}{v}\frac{\partial}{\partial\rho}\right)\left(e^{-\frac{y^2}{4v}\frac{1+\rho}{\rho}} p(g,0,v) \right) dv d\rho.$$
Observing that, as a function of $v$, $e^{-\frac{y^2}{4v}\frac{1+\rho}{\rho}} p(g,0,v)$ tends to $0$ both as $v\to 0^+$ and $v\to \infty$ thanks to \eqref{ge}, we deduce that
\begin{align*}
& \mathscr L^s (\mathfrak e^{(s)}(\cdot,y))(g) = \frac{(4\pi)^{s-1}}{\G(1-s)} \int_0^\infty\int_0^\infty \frac{\rho^{s}}{v}\frac{\partial}{\partial\rho}\left(e^{-\frac{y^2}{4\rho v}}e^{-\frac{y^2}{4v}} p(g,0,v) \right) dv d\rho \\
&=\frac{(4\pi)^{s}}{\G(1-s)} \int_0^\infty \frac{1}{4\pi v}e^{-\frac{y^2}{4v}} p(g,0,v)\left( \int_0^\infty \rho^{s}\frac{\partial}{\partial\rho}\left(e^{-\frac{y^2}{4\rho v}}\right) d\rho\right) dv \\
&= \frac{(4\pi)^{s}}{\G(1-s)} \int_0^\infty \frac{1}{4\pi v}e^{-\frac{y^2}{4v}} p(g,0,v)\left(\frac{y^2}{4v} \int_0^\infty \rho^{s-1}e^{-\frac{y^2}{4\rho v}} \frac{d\rho}{\rho}\right) dv.
\end{align*}
Finally, keeping in mind that
\[
\frac{y^2}{4v} \int_0^\infty \rho^{s-1}e^{-\frac{y^2}{4\rho v}} \frac{d\rho}{\rho}= \frac{\G(1-s)}{(4v)^{s}} y^{2s},
\]
and recalling
\begin{equation}\label{qmenos}
q^{(-s)}(g,v,y)  =   \frac{1}{(4\pi v)^{1+ s}}e^{-\frac{y^2}{4v}} p(g,e,v),
\end{equation}
we reach the desired conclusion \eqref{intrallnongeom}.
\end{proof}

In the Abelian case when $\bG = \Rn$, the functions $\mathfrak e^{(\pm s)}(g,y)$ coincide with $E^{(\pm s)}(x,y)$ in \eqref{Es}. However, we warn the reader that for a general non-Abelian group $\bG$ the conformal significance of formulas \eqref{Euqs} and \eqref{stic} is lost.  We close this section by recording that, as an immediate consequence of Theorem \ref{CHnon} and formulas \eqref{Euqs}-\eqref{stic}, we have the following.

\begin{corollary}\label{C:eu}
In $\Rn$, with $n\geq 2$, we have for $0<s<1$,
$$
\left(-\Delta\right)^s\left((|x|^2 + y^2)^{-\frac{n- 2s}2}\right)= \frac{\G\left(\frac{n}{2}+s\right)}{\G\left(\frac{n}{2}-s\right)}  \left(2 y\right)^{2s} (|x|^2 + y^2)^{-\frac{n+ 2s}2},\qquad x\in\R^n,\ y>0.
$$
\end{corollary}
\begin{proof}
We observe that \eqref{defq} and its counterpart obtained by changing $s$ into $-s$ presently give
\begin{equation}\label{defeu}
 q^{(\pm s)}(g,t,y)  =   \frac{1}{(4\pi t)^{1\mp s}}e^{-\frac{y^2}{4t}} \frac{1}{(4\pi t)^{\frac n2}} e^{-\frac{|x|^2}{4t}} = \frac{1}{(4\pi t)^{\frac{n}{2}+1\mp s}}e^{-\frac{|x|^2+y^2}{4t}}.
\end{equation}
From \eqref{defeu} and the definition \eqref{defe} we find
\begin{equation}\label{defeuu}
\mathfrak e^{(\pm s)}(g,y) \overset{def}{=} \int_0^\infty \frac{1}{(4\pi t)^{\frac{n}{2}+1\mp s}}e^{-\frac{|x|^2+y^2}{4t}} dt = \frac{\G(\frac{n\mp 2s}2)}{4 \pi^{\frac n2+1 \mp s}} (|x|^2 + y^2)^{-\frac{n\mp 2s}2}.
\end{equation}
Substituting the right-hand side of \eqref{defeuu} in \eqref{intrallnongeom} of Theorem \ref{CHnon} we reach the desired conclusion.

\end{proof}

\section{Conformal intertwining formula}\label{S:confi}

The objective of this section is to prove Theorem \ref{CHrevisited} following the ideas outlined in Sections \ref{S:intro} and \ref{S:nci}. We begin by introducing the relevant geometric framework. Let $\bG$ be a Carnot group of step $r=2$ with Lie algebra $\bg = V_1 \oplus V_2$, which we assume endowed with an inner product $\langle \cdot,\cdot\rangle$ and induced norm $|\cdot|$.  We fix orthonormal basis $\{e_1,...,e_{m}\}$ and $\{\ve_1,...,\ve_k\}$ for $V_1$ and $V_2$ respectively. Points $z\in V_1$ and $\sigma\in V_2$ will be identified with either one of the representations $z = \sum_{j=1}^{m} z_j e_j$, $\sigma = \sum_{\ell=1}^k \sigma_\ell \ve_\ell$, or also $z = (z_1,...,z_{m})$, $\sigma = (\sigma_1,...,\sigma_k)$. Accordingly, whenever convenient we will identify the point $g = \exp(z+\sigma)\in \bG$ with its logarithmic coordinates $(z,\sigma)\in \R^m\times\R^k$. The Kaplan mapping $J: V_2 \to \operatorname{End}(V_1)$ is defined by 
\begin{equation}\label{kap}
\langle J(\sigma)z,\zeta\rangle = \langle [z,\zeta],\sigma\rangle = - \langle J(\sigma)\zeta,z\rangle,
\end{equation}
for $z,\zeta\in V_1$ and $\sigma\in V_2$.
Clearly, $J(\sigma)^\star = - J(\sigma)$, and one has $<J(\sigma)z,z> = 0$. By \eqref{kap} and the Baker-Campbell-Hausdorff formula, see e.g. p. 12 of \cite{CoG}, 
$$\exp(z+\sigma)  \exp(\zeta+\tau) = \exp{\left(z + \zeta + \sigma +\tau + \frac{1}{2}
[z,\zeta]\right)},$$
we obtain the non-Abelian multiplication in $\bG$
\begin{equation}\label{grouplaw}
g\circ g' = (z,\sigma)\circ (\zeta,\tau) = \big(z +\zeta,\sigma + \tau + \frac 12 \sum_{\ell=1}^{k} <J(\ve_\ell)z,\zeta>\ve_\ell\big).
\end{equation}
From \eqref{grouplaw} it is easy to recognise that $e=(0,0)$ and $(z,\sigma)^{-1}=(-z,-\sigma)$. 

Henceforth in this section we assume that $\bG$ is of Heisenberg type. By this we mean that for every $\sigma \in V_2$ one has 
$$J(\sigma)^2 = - |\sigma|^2\ {\mathbb I}_{V_1}.$$
We now denote by $\mathscr L$ the horizontal Laplacian associated with the orthonormal basis $\{e_1,\ldots,e_m\}$ of the horizontal layer $V_1$, see the opening of Section \ref{S:nci}. 

We consider the extension problem \eqref{parext0} above: given a function $u\in C^\infty_0(\bG\times \R_t)$, find a function $U\in C^\infty(\bG\times  \R_t \times \R^+_y)$ such that
\begin{equation}\label{parext00}
\begin{cases}
\mathfrak P_{(s)} U \overset{def}{=} \frac{\p^2 U}{\p y^2}  + \frac{1-2s}y \frac{\p U}{\p y} + \frac{y^2}4 \Delta_\sigma U+ \mathscr L U - \frac{\p U}{\p t} = 0,\ \ \ \ \ \text{in}\ \bG\times  \R_t \times \R^+_y,
\\
U(g,t,0) = u(g,t).
\end{cases}
\end{equation}
For the Heisenberg group $\Hn$ the time-independent version of \eqref{parext00} was first introduced in the cited work of Frank et al. \cite{FGMT}. Subsequently, Roncal and Thangavelu considered the time-dependent problem \eqref{parext00} in their cited papers \cite{RTaim, RT}. Our approach is quite different from that in these works since, as we have mentioned above, these authors do not rely on \eqref{parext00} to establish \eqref{Gnoneu}, but instead use non-commutative Fourier analysis and group representation theory.

We continue to denote by $q_{(\pm s)}((z,\sigma),t,y)$ respectively the fundamental solution of $\mathfrak P_{(s)}$ introduced in \eqref{parBfs0}, and the function obtained from it by changing $s$ into $-s$. We now consider the kernels in the thin space $\bG\times (0,\infty)$ defined by 
\begin{equation}\label{conn}
\mathscr K_{(\pm s)}((z,\sigma),t) = (4\pi t)^{1\mp s} q_{(\pm s)}((z,\sigma),t,0).
\end{equation}
From \eqref{parBfs0} and the definition \eqref{conn} it is easy to see that the kernels $\mathscr K_{(\pm s)}$ have the following explicit expression
\begin{equation}\label{poissono}
\mathscr K_{(\pm s)}((z,\sigma),t) = \frac{2^k}{\left(4\pi t\right)^{\frac m2+k}}\int_{\R^k} e^{-\frac it \langle \sigma,\la\rangle} \left(\frac{|\la|}{\sinh |\la|}\right)^{\frac m2+1\mp s}  e^{- \frac{|z|^2}{4t} \frac{|\la|}{\tanh |\la|}}d\la,\ \ \ \ 0<s\le 1.
\end{equation}
It is worth noting here that if we take $s=1$ in \eqref{poissono} we obtain the Gaveau-Hulanicki-Cygan heat kernel in $\bG$ (we refer the reader to \cite{GTstep2} for a recent pde-based derivation of such kernel). With \eqref{poissono} in hands, and by slightly abusing the notation, we define
\[
\mathscr K_{(\pm s)}(g,g',t)= \mathscr K_{(\pm s)}(g^{-1}\circ g',t) = \mathscr K_{(\pm s)}((g')^{-1}\circ g,t).
\]
With such functions in hands, we next introduce two operators on $\Lp(\bG)$ by the formula
\begin{equation}\label{Pst}
P_{(\pm s),t} u(g) = \int_{\bG} \mathscr K_{(\pm s)}(g,g',t) u(g') dg'.
\end{equation}
As for $P_{(s),t}$, we mention that its \emph{raison d'\^{e}tre} is in the fact, which is one of the main results in \cite{GTfeel}, that the operator
\begin{equation}\label{conformalriesz}
\mathscr I_{(2s)} u(g) = \frac{1}{\G(s)} \int_0^\infty t^{s-1} P_{(s),t} u(g) dt
\end{equation}
provides the inverse of $\mathscr L_s$. The relevance of the operator $P_{(-s),t} $, instead, is underscored by the following formula which was proved in \cite{RT} 
\begin{equation}\label{RTs}
\mathscr L_s u(g) = - \frac{s}{\G(1-s)} \int_0^\infty \frac{1}{t^{1+s}} \left[P_{(-s),t} u(g) - u(g)\right] dt.
\end{equation}

This preliminary discussion brings us to the heart of the present section, the proof of Theorem \ref{CHrevisited}. Our plan is to proceed as closely as possible to the proof of Theorem \ref{CHnon}, but we immediately encounter some difficulties. To explain this point we mention that in the non-geometric setting of Section \ref{S:nci} there are two aspects that play a crucial role: (i) the same heat kernel $p(\cdot,t)$ occurs both in the expression \eqref{bala2} of $\mathscr L^s$ and in that of the function $\mathfrak e^{(s)}$; (ii) the Chapman-Kolmogorov identity enters crucially in the final equality of the key identity \eqref{chko}. Both facts fail to hold in the present conformal setting. A third more pervasive complication is represented by the very different nature of the fundamental solutions $q^{(\pm s)}$ and $q_{(\pm s)}$ of the parabolic extension problems for $\mathscr L^s$ and $\mathscr L_s$.

The proof of Theorem \ref{CHrevisited} rests on two preliminary lemmas in which we circumvent the difficulties listed above. In the first one we  establish a replacement of the semigroup property for the group convolution of the different kernels $q_{(-s)}(\cdot,t,0)$ and $q_{(s)}(\cdot,\tau,y)$ which respectively appear in $P_{(-s),t}$ and $\mathfrak e_{(s)}(\cdot,y)$). We mention that the case $y=0$ of the following lemma was proved in \cite[Lemma 4.2]{GTfeel}.

\begin{lemma}\label{kernallazzamentoy}
Fix $s\in (0,1)$, $g\in\bG$, $t,\tau>0$, and $y\geq 0$. Then, we have
\begin{align}\label{eccolaQ}
\int_{\bG}& q_{(-s)}((g')^{-1}\circ g,t,0) q_{(s)}(g',\tau,y)\, dg'=\int_{\R^k} e^{2\pi i \langle \sigma,\la\rangle}\left(\frac{|\la|}{2\sinh 2\pi \tau |\la|}\right)^{1-s}\times\\
&\times\left(\frac{|\la|}{2\sinh 2\pi t |\la|}\right)^{1+s}\left(\frac{|\la|}{2\sinh 2\pi (t+\tau) |\la|}\right)^{\frac m2} e^{-\frac \pi2 |z|^2 \frac{|\la|}{\tanh 2\pi (t+\tau) |\la|}}e^{-\frac \pi2 y^2 \frac{|\lambda|}{\tanh 2\pi \tau |\lambda|}}d\la .\nonumber
\end{align}
\end{lemma}

\begin{proof}
To establish \eqref{eccolaQ}, we fix $t, \tau>0$, $y\geq 0$, and $z\in\R^m$. By partial Fourier transform with respect to the vertical variable $\sigma\in \R^k$, we observe that the desired conclusion \eqref{eccolaQ} will be true if we can prove that for all $\lambda\in\R^k$ the following holds
\begin{align}\label{claimFourier}
\int_{\R^k} e^{-2\pi i \langle \sigma,\la\rangle}\int_{\bG} q_{(-s)}((g')^{-1}\circ g,t,0) q_{(s)}(g',\tau,y)\, dg'd\sigma=\left(\frac{|\la|}{2\sinh 2\pi \tau |\la|}\right)^{1-s}\times&\\
\times\left(\frac{|\la|}{2\sinh 2\pi t |\la|}\right)^{1+s}\left(\frac{|\la|}{2\sinh 2\pi (t+\tau) |\la|}\right)^{\frac m2} e^{-\frac \pi2 |z|^2 \frac{|\la|}{\tanh 2\pi (t+\tau) |\la|}}e^{-\frac \pi2 y^2 \frac{|\lambda|}{\tanh 2\pi \tau |\lambda|}}.&
\nonumber
\end{align}
The rest of the proof of the lemma will thus be devoted to establishing \eqref{claimFourier}. Recalling the definition of $q_{(\pm s)}$ in \eqref{parBfs0}, by a simple change of variable we see that
\begin{equation}\label{defaltq}
q_{(\pm s)}(g',\tau,y)  = \int_{\R^k} e^{2\pi i \langle \sigma',\la\rangle} \left(\frac{|\la|}{2\sinh 2\pi \tau |\la|}\right)^{\frac m2+1\mp s}  e^{-\frac \pi2 (|z'|^2+y^2) \frac{|\la|}{\tanh 2\pi \tau |\la|}}d\la.
\end{equation}
From the expression \eqref{grouplaw} of the group law, we also have
\begin{align}\label{convqs}
&q_{(\pm s)}((g')^{-1}\circ g,t,0) \\
&= \int_{\R^k} e^{2\pi i \left(\langle \sigma-\sigma',\la\rangle + \frac 12\left\langle J(\lambda)z,z' \right\rangle\right)} \left(\frac{|\la|}{2\sinh 2\pi t |\la|}\right)^{\frac m2+1\mp s}  e^{-\frac \pi2 |z-z'|^2 \frac{|\la|}{\tanh 2\pi t |\la|}}d\la.\nonumber
\end{align}
Next, we note that by \eqref{defaltq}, \eqref{convqs}, and applying twice the Fourier inversion formula, for every fixed $\la\in \R^k$ we can rewrite the left-hand side of \eqref{claimFourier} in the following way
\begin{align}\label{lhs}
&\int_{\R^k} e^{-2\pi i \langle \sigma,\la\rangle}\int_{\bG} q_{(-s)}((g')^{-1}\circ g,t,0) q_{(s)}(g',\tau,y)\, dg'd\sigma\\
&=\int_{\R^k}\int_{\R^k}\int_{\R^k}\int_{\bG} e^{-2\pi i \langle \sigma,\la\rangle} e^{2\pi i \langle \sigma -\sigma',\mu\rangle}e^{\pi i \langle J(\mu)z,z'\rangle} e^{2\pi i \langle \sigma',\omega\rangle}\left(\frac{|\mu|}{2\sinh 2\pi t |\mu|}\right)^{\frac m2+1+s}\times\nonumber\\  
&\times \left(\frac{|\omega|}{2\sinh 2\pi \tau |\omega|}\right)^{\frac m2+1-s} \times e^{-\frac \pi2 |z-z'|^2 \frac{|\mu|}{\tanh 2\pi t |\mu|}} e^{-\frac \pi2 (|z'|^2+y^2) \frac{|\omega|}{\tanh 2\pi \tau |\omega|}} d g' d\omega d\mu d\sigma\nonumber\\
&=\left(\frac{|\lambda|}{2\sinh 2\pi t |\lambda|}\right)^{\frac m2+1+s}\int_{\R^k}\int_{\bG} e^{-2\pi i \langle \sigma',\lambda\rangle}e^{\pi i \langle J(\lambda)z,z'\rangle} e^{2\pi i \langle \sigma',\omega\rangle}  \left(\frac{|\omega|}{2\sinh 2\pi \tau |\omega|}\right)^{\frac m2+1-s} \times\nonumber\\
&\times e^{-\frac \pi2 |z-z'|^2 \frac{|\lambda|}{\tanh 2\pi t |\lambda|}} e^{-\frac \pi2 (|z'|^2+y^2) \frac{|\omega|}{\tanh 2\pi \tau |\omega|}} d g' d\omega \nonumber\\
&=\left(\frac{|\lambda|}{2\sinh 2\pi t |\lambda|}\right)^{\frac m2+1+s}\left(\frac{|\lambda|}{2\sinh 2\pi \tau |\lambda|}\right)^{\frac m2+1-s}e^{-\frac \pi2 y^2 \frac{|\lambda|}{\tanh 2\pi \tau |\lambda|}}\times\nonumber\\
&\times\int_{\R^m} e^{\pi i \langle J(\lambda)z,z'\rangle} e^{-\frac \pi2 |z-z'|^2 \frac{|\lambda|}{\tanh 2\pi t |\lambda|}} e^{-\frac \pi2 |z'|^2 \frac{|\lambda|}{\tanh 2\pi \tau |\lambda|}} d z'.\nonumber
\end{align}

We now notice that the last integral is the same appearing in proof of \cite[Lemma 4.2]{GTfeel}: in fact, exploiting the properties of the $J$ map \eqref{kap}, and elementary manipulations of hyperbolic functions, in the final equality of formula (4.6) in \cite{GTfeel} we showed that
\begin{align}\label{presodala}
&\int_{\R^m} e^{\pi i \langle J(\lambda)z,z'\rangle} e^{-\frac \pi2 |z-z'|^2 \frac{|\lambda|}{\tanh 2\pi t |\lambda|}} e^{-\frac \pi2 |z'|^2 \frac{|\lambda|}{\tanh 2\pi \tau |\lambda|}} d z'\\
&=\left(\frac{|\la|}{2\sinh 2\pi (t+\tau) |\la|}\right)^{\frac m2} e^{-\frac \pi2 |z|^2 \frac{|\la|}{\tanh 2\pi (t+\tau) |\la|}}.\notag
\end{align}
Inserting \eqref{presodala} in \eqref{lhs} we obtain the desired conclusion \eqref{claimFourier}. This completes the proof of the lemma.

\end{proof}

The next lemma is purely technical and its significance will be clear in the subsequent discussion.

\begin{lemma}\label{L:cowboy}
Let $0<s<1$, $B, \mu>0$. We have
\begin{align}\label{superpippero}
&\int_0^\infty \left(\frac{\mu}{(1+\rho)\sinh \frac{\rho}{1+\rho}\mu}\right)^2  \left(\frac{\sinh \frac{\rho}{1+\rho}\mu}{\sinh \frac{\mu}{1+\rho}}\right)^s e^{-\frac{B\mu}{\tanh \frac{\rho}{1+\rho}\mu }}d\rho \\
&= B^{s-1}\left(\frac{\mu}{\sinh\mu}\right)^s\G(1-s) e^{-\frac{B\mu}{\tanh\mu }}.\notag
\end{align}
\end{lemma}

\begin{proof}
If we keep in mind the following identity
$$
\frac{1}{\tanh \mu} - \frac{1}{\tanh \frac{\rho}{1+\rho}\mu}= - \frac{1}{\sinh\mu} \frac{\sinh \frac{\mu}{1+\rho}}{\sinh \frac{\rho}{1+\rho}\mu},
$$
we have
{\allowdisplaybreaks
\begin{align*}
&e^{\frac{B\mu}{\tanh\mu }}\int_0^\infty \left(\frac{\mu}{(1+\rho)\sinh \frac{\rho}{1+\rho}\mu}\right)^2  \left(\frac{\sinh \frac{\rho}{1+\rho}\mu}{\sinh \frac{\mu}{1+\rho}}\right)^s e^{-\frac{B\mu}{\tanh \frac{\rho}{1+\rho}\mu }}d\rho\\
&=\int_0^\infty \left(\frac{\mu}{(1+\rho)\sinh \frac{\rho}{1+\rho}\mu}\right)^2  \left(\frac{\sinh \frac{\rho}{1+\rho}\mu}{\sinh \frac{\mu}{1+\rho}}\right)^s e^{-\frac{B\mu}{\sinh\mu}\frac{\sinh \frac{\mu}{1+\rho}}{\sinh \frac{\rho}{1+\rho}\mu}}d\rho.
\end{align*}}

In the last integral we now make the change of variable $\rho\to \tau= \frac{B\mu}{\sinh\mu}\frac{\sinh \frac{\mu}{1+\rho}}{\sinh \frac{\rho}{1+\rho}\mu}$, which yields $d\tau = - \frac{B\mu^2}{(1+\rho)^2\sinh^2 \frac{\rho}{1+\rho}\mu} d\rho$. We thus have
\begin{align*}
&e^{\frac{B\mu}{\tanh\mu }}\int_0^\infty \left(\frac{\mu}{(1+\rho)\sinh \frac{\rho}{1+\rho}\mu}\right)^2  \left(\frac{\sinh \frac{\rho}{1+\rho}\mu}{\sinh \frac{\mu}{1+\rho}}\right)^s e^{-\frac{B\mu}{\tanh \frac{\rho}{1+\rho}\mu }}d\rho\\
&=\frac{1}{B}\int_0^\infty  \left(\frac{B\mu}{\tau \sinh\mu }\right)^{s} e^{-\tau}d\tau= B^{s-1}\left(\frac{\mu}{\sinh\mu}\right)^s\G(1-s).
\end{align*}
This proves \eqref{superpippero}.

\end{proof}

We are now ready to present the proof of the main result in this note.

\begin{proof}[Proof of Theorem \ref{CHrevisited}]
For $u\in C^\infty(\bG)$ and suitably decay at infinity, and $g \in \bG$, one has
\begin{align}\label{baladeriva}
& \mathscr L_s u(g) = - \frac{s}{\G(1-s)} \int_0^\infty \tau^{-s-1}\int_0^\tau \frac{\partial}{\partial t}\left(P_{(-s),t}u (g)\right) dt d\tau
\\
&= - \frac{s }{\G(1-s)} \int_0^\infty \int_{t}^{\infty} \tau^{-s-1} \frac{\partial}{\partial t}\left(P_{(-s),t}u (g)\right) d\tau dt\nonumber
\\
&= - \frac{1}{\G(1-s)} \int_0^\infty t^{-s} \frac{\partial}{\partial t}\left(P_{(-s),t}u (g) \right) dt, \nonumber
\end{align}
which is the conformal counterpart of \eqref{bala2}. For $y>0$ we want to apply  \eqref{baladeriva} with the choice
$$u(\cdot)=\mathfrak e_{(s)}(\cdot,y).$$
By \eqref{conn}, \eqref{Pst} and \eqref{frake0}, we have
\begin{align*}
& \mathscr L_s (\mathfrak e_{(s)}(\cdot,y))(g)=- \frac{1}{\G(1-s)} \int_0^\infty t^{-s} \frac{\partial}{\partial t}\left(P_{(-s),t}(\mathfrak e_{(s)}(\cdot,y))(g) \right) dt\\
&=- \frac{1}{\G(1-s)} \int_0^\infty\int_0^\infty t^{-s} \frac{\partial}{\partial t}\left( \int_{\bG}(4\pi t)^{1+ s} q_{(- s)}((g')^{-1}\circ g,t,0)q_{(s)}(g',\tau,y) dg' \right)  d\tau dt.
\end{align*}
In view of Lemma \ref{kernallazzamentoy} we thus infer
{\allowdisplaybreaks
\begin{align*}
& \mathscr L_s (\mathfrak e_{(s)}(\cdot,y))(g)\\
&=\frac{-1}{\G(1-s)} \int_0^\infty\int_0^\infty\int_{\R^k} t^{-s} \frac{\partial}{\partial t}\left( (4\pi t)^{1+ s} e^{2\pi i \langle \sigma,\la\rangle}\left(\frac{|\la|}{2\sinh 2\pi \tau |\la|}\right)^{1-s}\left(\frac{|\la|}{2\sinh 2\pi t |\la|}\right)^{1+s}\times\right.\notag\\
&\times\left.\left(\frac{|\la|}{2\sinh 2\pi (t+\tau) |\la|}\right)^{\frac m2} e^{-\frac \pi2 |z|^2 \frac{|\la|}{\tanh 2\pi (t+\tau) |\la|}}e^{-\frac \pi2 y^2 \frac{|\lambda|}{\tanh 2\pi \tau |\lambda|}}  \right)  d\lambda d\tau dt\notag \\
&=- \frac{(4\pi )^{s-1}}{\G(1-s)} \int_0^\infty\int_0^\infty\int_{\R^k} \frac{e^{2\pi i \langle \sigma,\la\rangle}}{\tau} \left(\frac{\tau}{t}\right)^s \frac{\partial}{\partial t}\left(  \left(\frac{2\pi \tau|\la|}{\sinh 2\pi \tau |\la|}\right)^{1-s}\left(\frac{2\pi t|\la|}{\sinh 2\pi t |\la|}\right)^{1+s}\times\right.\notag\\
&\times\left.\left(\frac{|\la|}{2\sinh 2\pi (t+\tau) |\la|}\right)^{\frac m2} e^{-\frac \pi2 |z|^2 \frac{|\la|}{\tanh 2\pi (t+\tau) |\la|}}e^{-\frac \pi2 y^2 \frac{|\lambda|}{\tanh 2\pi \tau |\lambda|}}  \right)  d\lambda d\tau dt.\notag
\end{align*}}
Hence, the same change of variable as in \eqref{change} leads to the following identity
{\allowdisplaybreaks
\begin{align}\label{id2}
& \mathscr L_s (\mathfrak e_{(s)}(\cdot,y))(g)\\
&=\frac{-(4\pi )^{s-1}}{\G(1-s)} \int_{\R^k}\int_0^\infty\int_0^\infty e^{2\pi i \langle \sigma,\la\rangle}\frac{\rho^{s-1}}{1+\rho} \left(\frac{\partial}{\partial v} - \frac{\rho(1+\rho)}{v}\frac{\partial}{\partial\rho}\right)\left( \left(\frac{2\pi \frac{\rho}{1+\rho}v|\la|}{\sinh 2\pi \frac{\rho}{1+\rho}v |\la|}\right)^{1-s}\times\right.\notag\\
&\times\left.\left(\frac{2\pi \frac{v}{1+\rho}|\la|}{\sinh 2\pi\frac{v}{1+\rho} |\la|}\right)^{1+s}\left(\frac{|\la|}{2\sinh 2\pi v |\la|}\right)^{\frac m2} e^{-\frac \pi2 |z|^2 \frac{|\la|}{\tanh 2\pi v |\la|}}e^{-\frac \pi2 y^2 \frac{|\lambda|}{\tanh 2\pi \frac{\rho}{1+\rho}v |\lambda|}}  \right)   d\rho dv d\lambda.\notag
\end{align}}
{\allowdisplaybreaks
We notice that, for any $\lambda \in\R^k$ and $\rho>0$, the function
\begin{align*}
\left(\frac{2\pi \frac{\rho}{1+\rho}v|\la|}{\sinh 2\pi \frac{\rho}{1+\rho}v |\la|}\right)^{1-s}\left(\frac{2\pi \frac{v}{1+\rho}|\la|}{\sinh 2\pi\frac{v}{1+\rho} |\la|}\right)^{1+s}\left(\frac{|\la|}{2\sinh 2\pi v |\la|}\right)^{\frac m2} e^{-\frac \pi2 |z|^2 \frac{|\la|}{\tanh 2\pi v |\la|}}e^{-\frac \pi2 y^2 \frac{|\lambda|}{\tanh 2\pi \frac{\rho}{1+\rho}v |\lambda|}}
\end{align*}
converges to $0$ both as $v\to 0^+$ and $v\to \infty$ (this holds for any $z\in \R^m$ and $y>0$). Inserting this information in \eqref{id2} we deduce}
{\allowdisplaybreaks
\begin{align}\label{id3}
& \mathscr L_s (\mathfrak e_{(s)}(\cdot,y))(g)\\
&=\frac{(4\pi )^{s-1}}{\G(1-s)} \int_{\R^k}\int_0^\infty\int_0^\infty e^{2\pi i \langle \sigma,\la\rangle}\frac{\rho^{s}}{v}\frac{\partial}{\partial\rho}\left( \left(\frac{2\pi \frac{\rho}{1+\rho}v|\la|}{\sinh 2\pi \frac{\rho}{1+\rho}v |\la|}\right)^{1-s}\left(\frac{2\pi \frac{v}{1+\rho}|\la|}{\sinh 2\pi\frac{v}{1+\rho} |\la|}\right)^{1+s}\times\right.\notag\\
&\times\left.\left(\frac{|\la|}{2\sinh 2\pi v |\la|}\right)^{\frac m2} e^{-\frac \pi2 |z|^2 \frac{|\la|}{\tanh 2\pi v |\la|}}e^{-\frac \pi2 y^2 \frac{|\lambda|}{\tanh 2\pi \frac{\rho}{1+\rho}v |\lambda|}}  \right)   d\rho dv d\lambda\notag\\
&=\frac{(4\pi )^{s}}{\G(1-s)} \int_{\R^k}\int_0^\infty e^{2\pi i \langle \sigma,\la\rangle} \frac{1}{(4\pi v)^{\frac m2 +1}} \left(\frac{2\pi v|\la|}{\sinh 2\pi v |\la|}\right)^{\frac m2}e^{-\frac \pi2 |z|^2 \frac{|\la|}{\tanh 2\pi v |\la|}}\times \notag\\ 
&\times\left( \int_0^\infty \rho^{s}\frac{\partial}{\partial\rho}\left( \left(\frac{2\pi \frac{\rho}{1+\rho}v|\la|}{\sinh 2\pi \frac{\rho}{1+\rho}v |\la|}\right)^{1-s}\left(\frac{2\pi \frac{v}{1+\rho}|\la|}{\sinh 2\pi\frac{v}{1+\rho} |\la|}\right)^{1+s}e^{-\frac \pi2 y^2 \frac{|\lambda|}{\tanh 2\pi \frac{\rho}{1+\rho}v |\lambda|}}  \right)   d\rho\right) dv d\lambda.
\notag
\end{align}}
To deal with the integral in the last equality in \eqref{id3}, for $0<s<1$ and $\mu>0$ we now consider the function $h_{s,\mu}:(0,\infty)\to \R$ defined by
\begin{equation}\label{defhsm}
h_{s,\mu}(\rho)=\frac{\mu}{\sinh\mu}\left(\frac{\sinh \frac{\rho \mu}{1+\rho}}{\sinh \frac{ \mu}{1+\rho}}\right)^s\left[\frac{\mu}{\sinh\mu}\frac{\rho}{(1+\rho)^2}\left(\frac{\sinh \frac{\rho \mu}{1+\rho}}{\sinh \frac{ \mu}{1+\rho}} +2\cosh\mu + \frac{\sinh \frac{ \mu}{1+\rho}}{\sinh \frac{\rho \mu}{1+\rho}} \right)-1\right].
\end{equation}
This function was introduced in \cite[Section 4]{GTfeel}, and the main motivation behind it is in the formula 
\begin{equation}\label{derhsm}
h'_{s,\mu}(\rho)=\rho^{s}\frac{\partial}{\partial\rho}\left( \left(\frac{\frac{\rho}{1+\rho}\mu}{\sinh \frac{\rho}{1+\rho}\mu}\right)^{1-s}\left(\frac{\frac{\mu}{1+\rho}}{\sinh \frac{\mu}{1+\rho}}\right)^{1+s}\right),
\end{equation}
see \cite[equation (4.19)]{GTfeel}. This property was critical in showing that the operator $\mathscr I_{(2s)}$ in \eqref{conformalriesz} inverts $\mathscr L_s$. If in \eqref{defhsm} we now take 
$\mu = 2\pi v |\la|$, we obtain from \eqref{id3} 
{\allowdisplaybreaks
\begin{align}\label{new}
\mathscr L_s (\mathfrak e_{(s)}(\cdot,y))(g)=\frac{(4\pi )^{s}}{\G(1-s)} \int_{\R^k}\int_0^\infty e^{2\pi i \langle \sigma,\la\rangle} \frac{1}{(4\pi v)^{\frac m2 +1}} \left(\frac{2\pi v|\la|}{\sinh 2\pi v |\la|}\right)^{\frac m2}\times&
\\ 
\times e^{-\frac \pi2 |z|^2 \frac{|\la|}{\tanh 2\pi v |\la|}} \int_0^\infty \left(h'_{s,2\pi v |\la|}(\rho)e^{-\frac \pi2 y^2 \frac{|\lambda|}{\tanh 2\pi \frac{\rho}{1+\rho}v |\lambda|}} +\right.&\notag\\
+\left. \rho^{s}\left(\frac{2\pi \frac{\rho}{1+\rho}v|\la|}{\sinh 2\pi \frac{\rho}{1+\rho}v |\la|}\right)^{1-s}\left(\frac{2\pi \frac{v}{1+\rho}|\la|}{\sinh 2\pi\frac{v}{1+\rho} |\la|}\right)^{1+s}\frac{\partial}{\partial\rho}\left(e^{-\frac \pi2 y^2 \frac{|\lambda|}{\tanh 2\pi \frac{\rho}{1+\rho}v |\lambda|}}  \right)\right) d\rho dv d\lambda,\notag&
\end{align}
where in the last identity we have used \eqref{derhsm}. We also observe that from \eqref{defhsm} we have $h_{s,\mu}(\rho)\to 0$ as $\rho\to 0$ and $\rho \to \infty$ (see \cite[(4.18)]{GTfeel}). From \eqref{new} we thus find}
{\allowdisplaybreaks
\begin{align}\label{id4}
& \mathscr L_s (\mathfrak e_{(s)}(\cdot,y))(g)\\
&=\frac{(4\pi )^{s}}{\G(1-s)} \int_{\R^k}\int_0^\infty e^{2\pi i \langle \sigma,\la\rangle} \frac{1}{(4\pi v)^{\frac m2 +1}} \left(\frac{2\pi v|\la|}{\sinh 2\pi v |\la|}\right)^{\frac m2}e^{-\frac \pi2 |z|^2 \frac{|\la|}{\tanh 2\pi v |\la|}}\times\notag\\
&\times\int_0^\infty \frac{\partial}{\partial\rho}\left(e^{-\frac \pi2 y^2 \frac{|\lambda|}{\tanh 2\pi \frac{\rho}{1+\rho}v |\lambda|}}  \right)\times\notag\\
&\times\left(\rho^{s}\left(\frac{2\pi \frac{\rho}{1+\rho}v|\la|}{\sinh 2\pi \frac{\rho}{1+\rho}v |\la|}\right)^{1-s}\left(\frac{2\pi \frac{v}{1+\rho}|\la|}{\sinh 2\pi\frac{v}{1+\rho} |\la|}\right)^{1+s} - h_{s,2\pi v |\la|}(\rho)\right)   d\rho  dv d\lambda\notag\\
&=\frac{4^s \pi^{s+1} y^2}{\G(1-s)} \int_{\R^k}\int_0^\infty e^{2\pi i \langle \sigma,\la\rangle} \frac{1}{(4\pi v)^{\frac m2 +2}} \left(\frac{2\pi v|\la|}{\sinh 2\pi v |\la|}\right)^{\frac m2}e^{-\frac \pi2 |z|^2 \frac{|\la|}{\tanh 2\pi v |\la|}}\times\notag\\
&\times\int_0^\infty e^{-\frac \pi2 y^2 \frac{|\lambda|}{\tanh 2\pi \frac{\rho}{1+\rho}v |\lambda|}}\left(\frac{2\pi  v |\la|}{(1+\rho)\sinh 2\pi \frac{\rho}{1+\rho}v |\la|}\right)^2\times\notag\\
&\times\left(\rho^{s}\left(\frac{2\pi \frac{\rho}{1+\rho}v|\la|}{\sinh 2\pi \frac{\rho}{1+\rho}v |\la|}\right)^{1-s}\left(\frac{2\pi \frac{v}{1+\rho}|\la|}{\sinh 2\pi\frac{v}{1+\rho} |\la|}\right)^{1+s} - h_{s,2\pi v |\la|}(\rho)\right)   d\rho dv d\lambda\notag.
\end{align}}
From the identity
\begin{align*}
\sinh^2 2\pi v |\la|=&\sinh^2 2\pi \frac{\rho}{1+\rho}v |\la| + \sinh^2 2\pi \frac{v}{1+\rho} |\la| +\\
&+ 2 (\cosh 2\pi v |\la|)(\sinh 2\pi \frac{\rho}{1+\rho}v |\la|)(\sinh 2\pi \frac{v}{1+\rho} |\la|) 
\end{align*}
and the definition \eqref{defhsm}, a straightforward computation shows that
$$
\rho^{s}\left(\frac{2\pi \frac{\rho}{1+\rho}v|\la|}{\sinh 2\pi \frac{\rho}{1+\rho}v |\la|}\right)^{1-s}\left(\frac{2\pi \frac{v}{1+\rho}|\la|}{\sinh 2\pi\frac{v}{1+\rho} |\la|}\right)^{1+s} - h_{s,2\pi v |\la|}(\rho) = \frac{2\pi v|\la|}{\sinh 2\pi v |\la|} \left(\frac{\sinh 2\pi \frac{\rho}{1+\rho}v |\la|}{\sinh 2\pi \frac{v}{1+\rho} |\la|}\right)^s.
$$
From \eqref{id4} we thus infer
\begin{align}\label{id5}
& \mathscr L_s (\mathfrak e_{(s)}(\cdot,y))(g)\\
&=\frac{4^s \pi^{s+1} y^2}{\G(1-s)}\int_0^\infty\int_{\R^k} e^{2\pi i \langle \sigma,\la\rangle} \frac{1}{(4\pi v)^{\frac m2 +2}} \left(\frac{2\pi v|\la|}{\sinh 2\pi v |\la|}\right)^{\frac m2 + 1}e^{-\frac \pi2 |z|^2 \frac{|\la|}{\tanh 2\pi v |\la|}} \times\notag\\
&\times\left(\int_0^\infty \left(\frac{2\pi  v |\la|}{(1+\rho)\sinh 2\pi \frac{\rho}{1+\rho}v |\la|}\right)^2  \left(\frac{\sinh 2\pi \frac{\rho}{1+\rho}v |\la|}{\sinh 2\pi \frac{v}{1+\rho} |\la|}\right)^s e^{-\frac \pi2 y^2 \frac{|\lambda|}{\tanh 2\pi \frac{\rho}{1+\rho}v |\lambda|}}d\rho\right) d\lambda dv \notag.
\end{align}
For any $\lambda\neq 0$ and $v>0$ we now apply Lemma \ref{L:cowboy} with the choices $\mu=2\pi v |\la|$ and $B=\frac{y^2}{4v}$, and find
\begin{align}\label{ehiehiehi}
&\int_0^\infty \left(\frac{2\pi  v |\la|}{(1+\rho)\sinh 2\pi \frac{\rho}{1+\rho}v |\la|}\right)^2  \left(\frac{\sinh 2\pi \frac{\rho}{1+\rho}v |\la|}{\sinh 2\pi \frac{v}{1+\rho} |\la|}\right)^s e^{-\frac \pi2 y^2 \frac{|\lambda|}{\tanh 2\pi \frac{\rho}{1+\rho}v |\lambda|}}d\rho \\
&= \frac{y^{2s-2}}{(4 v)^{s-1}}\left(\frac{2\pi v|\la|}{\sinh 2\pi v |\la|}\right)^{s}\G(1-s)e^{-\frac \pi2 y^2 \frac{|\la|}{\tanh 2\pi v |\la|}} \notag. 
\end{align}
Inserting \eqref{ehiehiehi} in \eqref{id5}, we finally obtain
\begin{align*}
& \mathscr L_s (\mathfrak e_{(s)}(\cdot,y))(g)\\
&=(2\pi)^{2s} y^{2s}\int_0^\infty\int_{\R^k} e^{2\pi i \langle \sigma,\la\rangle} \frac{1}{(4\pi v)^{\frac m2 +1+s}} \left(\frac{2\pi v|\la|}{\sinh 2\pi v |\la|}\right)^{\frac m2 + 1+s}e^{-\frac \pi2 (|z|^2+y^2) \frac{|\la|}{\tanh 2\pi v |\la|}} d\lambda dv \notag\\
&=(2\pi y)^{2s}\int_0^\infty q_{(- s)}((z,\sigma),v,y) dv = (2\pi y)^{2s} \mathfrak e_{(- s)}((z,\sigma),y),
\end{align*}
which gives the desired conclusion \eqref{confCH}.

\end{proof}

Although it should at this point be self-evident, for the sake of completeness we show how \eqref{Gnoneu} can be derived from Theorem \ref{CHrevisited} and Theorem \hyperref[A]{A}.

\begin{corollary}\label{C:chsect}
Let $\bG$ be of Heisenberg type, and let $0<s<1$. For every $(z,\sigma)\in\bG$, and $y>0$ one has
\begin{align*}
& \mathscr L_s\left(((|z|^2+y^2)^2+16|\sigma|^2)^{-\frac{m + 2k - 2s}4}\right) 
\\
&=\frac{\G\left(\frac{m+2+2s}{4}\right)\G\left(\frac{m+2k+2s}{4}\right)}{\G\left(\frac{m+2-2s}{4}\right)\G\left(\frac{m+2k-2s}{4}\right)} (4y)^{2s}((|z|^2+y^2)^2+16|\sigma|^2)^{-\frac{m + 2k + 2s}4}.
\end{align*}
\end{corollary}
\begin{proof}
Combining \eqref{qmasomeno0} and \eqref{confCH}, we obtain
\begin{align}\label{transferch}
& \mathscr L_s\left(((|z|^2+y^2)^2+16|\sigma|^2)^{-\frac{m + 2k - 2s}4}\right) 
\\
&= \frac{(4\pi)^{1-s}}{\G(s) C_{(s)}(m,k)} \mathscr L_s \left(\mathfrak e_{(s)}((z,\sigma),y) \right)= \frac{(4\pi)^{1-s}(2\pi y)^{2s}}{\G(s) C_{(s)}(m,k)}\mathfrak e_{(-s)}((z,\sigma),y)\notag\\
&=\frac{y^{2s}}{4^s} \frac{|\G(-s)| C_{(-s)}(m,k)}{\G(s) C_{(s)}(m,k)} ((|z|^2+y^2)^2+16|\sigma|^2)^{-\frac{m + 2k + 2s}4}.\notag
\end{align}
On the other hand, by \eqref{C} we have
\begin{equation}\label{ratiocsmk}
\frac{|\G(-s)| C_{(-s)}(m,k)}{\G(s) C_{(s)}(m,k)} = 4^{3s} \frac{\G\left(\frac{m+2+2s}{4}\right) \G\left(\frac{m + 2k +2s}{4}\right)}{ \G\left(\frac{m+2-2s}{4}\right) \G\left(\frac{m + 2k -2s}{4}\right)}.
\end{equation}
Substituting \eqref{ratiocsmk} in \eqref{transferch} we reach the sought for conclusion.

\end{proof}

In closing, we observe that it is clear from Corollary \ref{C:chsect} that, for any $y>0$, the function
\begin{equation}\label{uy}
u_y(z,\sigma)= \left(\frac{\G\left(\frac{m+2+2s}{4}\right)\G\left(\frac{m+2k+2s}{4}\right)}{\G\left(\frac{m+2-2s}{4}\right)\G\left(\frac{m+2k-2s}{4}\right)}\right)^{\frac{m + 2k - 2s}{4s}}\left(\frac{16y^2}{(|z|^2+y^2)^2+16|\sigma|^2}\right)^{\frac{m + 2k - 2s}4}
\end{equation}
satisfies $\mathscr L_s u_y = u_y^{\frac{m+2k+2s}{m+2k-2s}}$. It is worth emphasising here that Corollary \ref{C:chsect} is stable under the convergence of $s\nearrow 1$. In particular, we recover from \eqref{uy} the functions that, in the local case $s=1$, were found in \cite{JL} in the Heisenberg group $\Hn$, see also \cite[Theor. 1.1]{GVma} and \cite{GV} for partial results in groups of Heisenberg type.





\bibliographystyle{amsplain}

\begin{thebibliography}{10}

\bibitem{Bala}
A. V. Balakrishnan, 
\emph{An operational calculus for infinitesimal generators of semigroups}. Trans. Amer. Math. Soc. \textbf{91}~(1959), 330-353.

\bibitem{BJ}
H. Baum \& A. Juhl, \emph{Conformal differential geometry. $Q$-curvature and conformal holonomy}. Oberwolfach Seminars, 40, Birkh\"{a}user Verlag, Basel, 2010.

\bibitem{B}
T. P. Branson, \emph{Sharp inequalities, the functional determinant, and the complementary series}. Trans. Amer. Math. Soc. \textbf{347}~(1995), no. 10, 3671-3742.

\bibitem{Bsur}
T. P. Branson, \emph{Spectral theory of invariant operators, sharp inequalities, and representation theory}. The {P}roceedings of the 16th {W}inter {S}chool ``{G}eometry and {P}hysics'' ({S}rn\'{\i}, 1996). Rend. Circ. Mat. Palermo (2) Suppl. (1997), no. 46, 29-54.

\bibitem{BFM}
T. P. Branson, L. Fontana \& C. Morpurgo, \emph{Moser-Trudinger and Beckner-Onofri's inequalities on the $\operatorname{CR}$ sphere}. Ann. of Math. (2) \textbf{177}~(2013), no. 1, 1-52.


\bibitem{CS}
L. Caffarelli \& L. Silvestre, \emph{An extension problem related to the fractional Laplacian}. Comm. Partial Differential Equations \textbf{32}~(2007), no. 7-9, 1245-1260.

\bibitem{CG}
S.-Y. A. Chang \& M. d. M. Gonz\'alez, \emph{Fractional Laplacian in conformal geometry}. Adv. Math. \textbf{226}~(2011), no. 2, 1410-1432. 


\bibitem{CLO}
W. Chen, C. Li \& B. Ou, \emph{Classification of solutions for an integral equation}. Comm. Pure Appl. Math. \textbf{59}~(2006), no. 3, 330-343.

\bibitem{CoG}
 L. Corwin \& F. P. Greenleaf, \emph{Representations of nilpotent Lie groups and their applications,
Part I: basic theory and examples}.
 Cambridge Studies in Advanced Mathematics 18,
Cambridge University Press, Cambridge
(1990).

\bibitem{Co}
M. Cowling, \emph{Unitary and uniformly bounded representations of some simple Lie groups}. In ``Harmonic analysis and group representations'', A. Fig\`a Talamanca (Ed.), C.I.M.E. Summer Schools, Cortona, Italy, 1980.

\bibitem{CDKR}
M. Cowling, A. H. Dooley, A. Kor\'anyi \& F. Ricci, \emph{$H$-type groups and Iwasawa decompositions}. Adv. Math. \textbf{87}~(1991), no. 1, 1-41.


\bibitem{CH}
M. Cowling \& U. Haagerup, \emph{Completely bounded multipliers of the Fourier algebra of a simple Lie group of real rank one}. Invent. Math. \textbf{96}~(1989), no. 3, 507-549. 

\bibitem{FF} 
F. Ferrari \& B. Franchi, \emph{Harnack inequality for fractional sub-Laplacians in Carnot groups}. Math. Z. \textbf{279}~(2015), no. 1-2, 435-458.

\bibitem{Fo75}
G. B. Folland, \emph{Subelliptic estimates and function spaces on nilpotent Lie groups}. Ark. Mat. \textbf{13}~(1975), no. 2, 161-207. 

\bibitem{FL}
R. L. Frank \& E. H. Lieb, \emph{Sharp constants in several inequalities on the Heisenberg group}. Ann. of Math. (2) \textbf{176}~(2012), no. 1, 349-381.

\bibitem{FGMT}
R. L. Frank, M. d. M. Gonz\'alez, D. Monticelli \& J. Tan, \emph{An extension problem for the $CR$ fractional Laplacian}. Adv. Math. \textbf{270}~(2015), 97-137.

\bibitem{GVma}
N. Garofalo \& D. Vassilev, \emph{Regularity near the characteristic set in the non-linear Dirichlet problem and conformal geometry of sub-Laplacians on Carnot groups}. Math. Ann. \textbf{318}~(2000), no. 3, 453-516.

\bibitem{GV}
N. Garofalo \& D. Vassilev, \emph{Symmetry properties of positive entire solutions of Yamabe-type equations on groups of Heisenberg type}. Duke Math. J. \textbf{106}~(2001), no. 3, 411-448.

\bibitem{Gejde}
N. Garofalo, \emph{Some properties of sub-Laplaceans}. Proceedings of the International Conference "Two nonlinear days in Urbino 2017", 103-131, Electron. J. Differ. Equ. Conf., 25, Texas State Univ.-San Marcos, Dept. Math., San Marcos, TX, 2018. 

\bibitem{Gft}
N. Garofalo, \emph{Fractional thoughts}. New developments in the analysis of nonlocal operators, 1-135, Contemp. Math., 723, Amer. Math. Soc., Providence, RI, 2019.

\bibitem{GTfeel}
N. Garofalo \& G. Tralli, \emph{Feeling the heat in a group of Heisenberg type}. Adv. Math. \textbf{381}~(2021), 107635.

\bibitem{GTstep2}
N. Garofalo \& G. Tralli, \emph{Mehler met Ornstein and Uhlenbeck: the geometry of Carnot groups of step two and their heat kernels}.  ArXiv preprint 2007.10862.

\bibitem{dmG}
M. d. M. Gonz\'{a}lez, \emph{Recent progress on the fractional Laplacian in conformal geometry}. Recent developments in nonlocal theory, 236-273, De Gruyter, Berlin, 2018.

\bibitem{GMM}
C. Guidi, A. Maalaoui \& V. Martino, \emph{Palais-Smale sequences for the fractional CR Yamabe functional and multiplicity results}. Calc. Var. Partial Differential Equations \textbf{57}~(2018), article no. 152.

\bibitem{Ho}
L. H{\"o}rmander,
\textit{Hypoelliptic second order differential equations}. Acta Math. \textbf{119}~(1967), 147-171.

\bibitem{JL}
D. Jerison \& J. M. Lee, \emph{The Yamabe problem on CR manifolds}. J. Differential Geom. \textbf{25}~(1987), no. 2, 167-197. 

\bibitem{JL2}
D. Jerison \& J. M. Lee, \emph{Extremals for the Sobolev inequality on the Heisenberg group and the CR Yamabe problem}. J. Amer. Math. Soc. \textbf{1}~(1988), no. 1, 1-13.

\bibitem{Ka}
A. Kaplan, \emph{Fundamental solutions for a class of hypoelliptic PDE generated by composition of quadratic forms}. Trans. Amer. Math. Soc. \textbf{258}~(1980), no. 1, 147-153.

\bibitem{KS}
A. W. Knapp \& E. M. Stein, 
\emph{Intertwining operators for semisimple groups}. 
Ann. of Math. (2) \textbf{93}~(1971), 489-578. 

\bibitem{Kr}
A. Krist\'{a}ly, 
\emph{Nodal solutions for the fractional Yamabe problem on Heisenberg groups}. 
Proc. Roy. Soc. Edinburgh Sect. A \textbf{150}~(2020), 771-788. 

\bibitem{L}
E. H. Lieb, \emph{Sharp constants in the Hardy-Littlewood-Sobolev and related inequalities}. Ann. of Math. (2) \textbf{118}~(1983), no. 2, 349-374.

\bibitem{R}
M. Riesz, \emph{Int\'egrales de Riemann-Liouville et potentiels}. Acta Sci. Math. Szeged, \textbf{9}~(1938), 1-42.

\bibitem{RTaim}
L. Roncal \& S. Thangavelu, \emph{Hardy's inequality for fractional powers of the sublaplacian on the Heisenberg group}. Adv. Math. \textbf{302}~(2016), 106-158.

\bibitem{RT}
L. Roncal \& S. Thangavelu, \emph{An extension problem and trace Hardy inequality for the sublaplacian on $H$-type groups}. Int. Math. Res. Not. IMRN (2020), no. 14, 4238-4294.

\bibitem{St}
E. M. Stein, \emph{Singular integrals and differentiability properties of functions}. Princeton Mathematical Series, No. 30 Princeton University Press, Princeton, N.J. 1970 xiv+290 pp. 

\bibitem{VSC}
N. Th. Varopoulos, L. Saloff-Coste \& T. Coulhon, \emph{Analysis and geometry on groups}. Cambridge Tracts in Mathematics, 100. Cambridge University Press, Cambridge, 1992. xii+156 pp.


\end{thebibliography}

\end{document}